\documentclass{IEEEtran}
\usepackage{cite}
\usepackage{amsmath,amssymb,amsfonts,amsthm}
\usepackage{algorithmic}
\usepackage{graphicx}
\usepackage{algorithm,algorithmic}
\usepackage{hyperref}
\hypersetup{hidelinks=true}
\usepackage{textcomp}
\usepackage{epstopdf}
\usepackage{epsfig}
\usepackage{tikz, pgfplots}
\usepackage{booktabs}
\usepackage{longtable}
\usepackage{tabularx}
\usepackage{lipsum}  
\usepackage{multirow}
\usepackage{diagbox}
\usepackage{float}
\usepackage{subfig}
\usepackage{cancel}
\newtheorem{assum}{Assumption}
\newtheorem{lem}{Lemma}
\newtheorem{thm}{Theorem}

\newtheorem{rem}{Remark}

\newtheorem{defn}{Definition}

\def\BibTeX{{\rm B\kern-.05em{\sc i\kern-.025em b}\kern-.08em
    T\kern-.1667em\lower.7ex\hbox{E}\kern-.125emX}}
\begin{document}
\title{Multi/Single-stage structured zero-gradient-sum approach for prescribed-time optimization}
\author{Shuaiyu Zhou, Yiheng Wei, Jinde Cao, and Yang Liu
\thanks{Shuaiyu Zhou, Yiheng Wei and Jinde Cao are with the School of Mathematics, Southeast University, Nanjing 211189, China. J.D. Cao is also with the Yonsei Frontier Lab, Yonsei University, Seoul 03722, South Korea (e-mail: sy\_zhou@seu.edu.cn; neudawei@seu.edu.cn; jdcao@seu.edu.cn). }
\thanks{Yang Liu is with the Key Laboratory of Intelligent Education Technology and Application of Zhejiang Province, School of Mathematical Sciences, Zhejiang Normal University, Jinhua 321004, China (e-mail: liuyang@zjnu.edu.cn)}
}
\maketitle

\begin{abstract}
Prescribed-time convergence mechanism has become a prominent research focus in the current field of optimization and control due to its ability to precisely control the target completion time. The recently arisen prescribed-time algorithms for distributed optimization, currently necessitate multi-stage structures to achieve global convergence. This paper introduces two modified zero-gradient-sum algorithms, each based on a multi-stage and a single-stage structural frameworks established in this work. These algorithms are designed to achieve prescribed-time convergence and relax two common yet stringent conditions. This work also bridges the gap in current research on single-stage structured PTDO algorithm. The excellent convergence performance of the proposed algorithms is validated through a case study. 
\end{abstract}


\section{Introduction}\label{Section 1}
Distributed optimization, focusing on achieving global objectives through cooperation between multi-agent systems, has been extensively researched in the past decade \cite{Yang:2019ARC,Huang:2022Automatica}. Existing studies primarily tackled distributed optimization for convex functions, particularly gradient-descent based methods \cite{Nedic:2009TAC,Zhu:2011TAC,Nedic:2014TAC}. However, their reliance on local gradients necessitated diminishing step-sizes for global exact convergence, leading to slower convergence rates. Subsequently, new algorithms, such as proportional-integral based methods\cite{Kia:2015Automatica,Wang:2023TAC}, gradient tracking \cite{Nedic:2017SIAM,Xi:2017TAC,Xin:2019TAC,Shi:2023TAC}, and zero-gradient-sum (ZGS) \cite{Lu:2012TAC,Wu:2021TSMCS,Chen:2023ISAS}, were introduced to achieve global exact convergence with fixed step-sizes.

However, most of these algorithms have unbounded theoretical convergence times, limiting their practicality in those engineering problems with time constraints. In recent years, a new category of distributed optimization algorithms has emerged, guaranteeing bounded convergence times. Initial efforts included finite-time convergence distributed algorithms \cite{Polyakov:2015Auto,Wang:2020TAC} based on non-smooth formulations, but they still had unbounded bounds associated with initial states of the system. To address this, fixed-time convergence algorithms were developed \cite{Chen:2018Auto,Wang:2020Auto,Song:2021TNSE}, providing convergence time bounds independent of initial states but reliant on system control parameters. However, establishing a direct relationship between parameter selection and desired convergence time still remains challenging \cite{Chen:2023IJSS}. Recently, optimization research has centered on developing prescribed-time algorithms, which enable the deliberate setting of convergence times without relying on additional parameters. While initially applied in fields like system control \cite{Ni:2020TSMCS}, they've only recently been integrated into distributed optimization algorithms. Presently, prescribed-time algorithms for distributed optimization (PTDO) take three primary forms: non-smooth formulation, time-base generator functions \cite{Guo:2020TSMCS}, and time-varying scaling functions \cite{Wang:2018TC}. Of these, the latter two, known for their innovation and efficiency, have gained more attention.

The goal of this paper is to construct high-performance and user-friendly ZGS algorithms for PTDO in multi-stage/single-stage structures based on the aforementioned time-varying scaling functions. The motivations behind this research can be summarized into the following two main aspects,
\begin{itemize}
	\item [1)] current multi-stage prescribed-time ZGS methods involve separate stage executions for tasks like `local minimization' and `global convergence', which may introduce additional complexity and overhead. Hence, exploring multi-stage ZGS algorithms without explicit stage separation is valuable.
	\item [2)] a research gap currently exists regarding single-stage prescribed-time ZGS algorithms. This is typically attributed to the initial value selection for ZGS, frequently requiring an extra stage in algorithm design or analysis to address this concern. Thus, developing a single-stage ZGS algorithm that ensures both `initial condition satisfaction' and `global prescribed-time convergence' in a single time interval  addresses this gap.
\end{itemize}

The contributions of this paper are threefold.
\begin{itemize}
	\item [1)] \textit{Multi/Single-Stage Structure}: This paper conducts a comprehensive study and review of recent ZGS algorithms with different convergence rates, providing a unified framework to incorporate these algorithms within the proposed multi/single-stage structures.
	\item [2)] \textit{Multi-Stage ZGS Algorithm}: An implicit multi-stage ZGS algorithm is proposed in this paper, allowing for prescribed-time convergence without requiring any initial value constraints. The incorporation of the implicit multi-stage structure enhances the algorithm's applicability.
	\item [3)] \textit{Single-Stage ZGS Algorithm}: An novel single-stage ZGS algorithm is proposed in this paper, enabling prescribed-time convergence without the need for initial value constraint. Importantly, this algorithm doesn't demand compliance with another typical constraint found in ZGS algorithms. Unlike many related studies, this algorithm is single-stage in both protocol design and convergence analysis, bypassing the prerequisite of satisfying initial value conditions before conducting convergence analysis.
\end{itemize}

The rest of this paper is organized as follows. Section \ref{Section 2} presents the objective problem and introduces several fundamental concepts. In Section \ref{Section 3}, the multi/single-stage structure and two novel algorithms are introduced, and their prescribed-time convergence properties are analyzed. Section \ref{Section 4} contains the simulation results that illustrate the performance of the proposed algorithms. Lastly, Section \ref{Section 5} provides the concluding remarks and reflections.

\textit{Notation:} Throughout the paper, $\|\cdot\|$  denotes the 2-norm for vectors, and the Frobenius norm for matrices. $\nabla^{2} f$  denotes the Hessian matrix of the twice differentiable function  $f$. $I_{N}$ is the identity matrices of $N\times N$ and $\mathbf{1}_{N}$ is a vector of ones. $\otimes$ represents the Kronecker product. $\mathcal{N}_{i}$ is the set of neighbor nodes of agent $i$ in a multi-agent system. $L^{\top}$ means the transpose of matrix $L$. $\chi^{(i)},i=0,1,2$ represents the $i$th order derivative of $\chi$.

\section{Preliminaries}\label{Section 2}

\subsection{Prescribed-time stability}
\textbf{Time-varying scaling function} The following definition of time-varying scaling function is first given, which is the core of the design of a distributed optimizaiton with prescribed-time convergence performance.
\begin{defn}\cite{Wang:2018TC,Chen:2020TC}\label{def1}
	The time-varying scaling functions used here are defined as 
	\begin{equation}\label{mu1}
		\varrho_{1}(t)\triangleq \varrho_{1}(t;t_{0},T_{1})=\left\lbrace
		\begin{array}{cl}
			\frac{T_{1}^{h_{1}}}{\left(T_{1}+t_{0}-t\right)^{h_{1}}}, &\hspace{-6pt}\quad t \in\left[t_{0}, t_{0}+T_{1}\right),\\
			1 , &\hspace{-6pt}\quad otherwise,
		\end{array}
		\right.
	\end{equation}
	\begin{equation}\label{mu2}
		\varrho_{2}(t)\triangleq \varrho_{1}(t;t_{1},T_{2})=\left\lbrace
		\begin{array}{cl}
			\frac{T_{2}^{h_{2}}}{\left(T_{2}+t_{1}-t\right)^{h_{2}}}, &\hspace{-6pt}\quad t \in\left[t_{1}, t_{1}+T_{2}\right),\\
			1 , &\hspace{-6pt}\quad otherwise,
		\end{array}
		\right.
	\end{equation}
	where $h_{1},h_{2}>0$, $T_{1},T_{2}>0$ usually represent the desired prescribed times and the duration of the time-scaling function, $t_{0}\ge 0$ and $t_{1}=t_{0}+T_{1}$ are the initial instant of each duration.
\end{defn}

\textbf{Prescribed-time stable} The following definition and lemma regarding prescribed-time stability are essential for the algorithm analysis in this paper.
\begin{defn}\cite{Gong:2021TNSE}\label{def2}
	Considering a Lyapunov function $V(x,t)$, if the following equations hold,
	\begin{equation}\label{PTC}
		\begin{array}{rl}
			1)&\hspace{-6pt}  \lim\limits _{t \rightarrow\left(t_{0}+T\right)^{-}} V(x, t)=0 , \\
			2)&\hspace{-6pt}  V(x, t) \equiv 0, \forall t \in\left[t_{0}+T, \infty\right) ,
		\end{array}
	\end{equation}
	where $t_{0}$ is the initial instant, then $V(x,t)$ is globally prescribed-time stable with prescribed time interval $T$. 
\end{defn}
\begin{lem}\cite[Lemma 1]{Wang:2018TC}\label{lem7}
	Considering a Lyapunov function $V(x,t)$, if there exists $\kappa > 0$ such that
	\begin{equation}\nonumber
		\begin{array}{l}
			\dot{V}(x, t)=-\left[\kappa+2 \frac{\dot{\varrho}\left(t ; t_{0}, T\right)}{\varrho\left(t ; t_{0}, T\right)}\right] V(x, t), t \in\left[t_{0}, \infty\right),
		\end{array}
	\end{equation}
	then it can be derived that $V(x, t) = \varrho\left(t ; t_{0}, T\right)^{-2} \mathrm{e}^{\left(-\kappa\left(t-t_{0}\right)\right)}  V\left(x\left(t_{0}\right), t_{0}\right),t\in [t_{0},+\infty)$, which satisfies the conditions \eqref{PTC} in Definition \ref{def2}. Then $V(x,t)$ is globally prescribed-time stable within $T$. 
\end{lem}

\subsection{Graph theory}
A undirected graph $\mathcal{G}=(\mathcal{V},\mathcal{E},\mathcal{A})$ is defined by a node set $\mathcal{V}=\left\lbrace v_1,v_2,...,v_N\right\rbrace $, an undirected edge set $\mathcal{E}\subseteq\mathcal{V}\times\mathcal{V}$, and an adjacency matrix $\mathcal{A}=\left[ a_{ij}\right] _{N\times N}$ with non-negative elements. The Laplacian matrix $L$ is defined as $D-A$ where $D={\rm diag}\{d(v_1),d(v_2),\ldots,d(v_N)\}$ is the degree matrix. The graph $\overline{\mathcal{G}}$ is the complete graph of $\mathcal{G}$ if it has the same number of agents and contains an edge between every pair of nodes.

The following lemmas provide the properties of undirected graphs.
\begin{lem}\cite{Yu:2021ACC}\label{lem0}
	Given a connected and undirected graph $\mathcal{G}$ with $N$ nodes, it follows that
	\begin{itemize}
		\item [1)] $L$ is semi-positive definite and $\mathbf{1}_{N}^{\top}L=0$, $L\mathbf{1}_{N}=0$ hold.
		\item [2)] $L$ has one zero eigenvalue, while the remaining eigenvalues possess nonnegative real parts which can be arranged in ascending order as $\left[0, \lambda_1(L), \ldots, \lambda_{N}(L)\right]$.
		\item [3)] a positive definite matrix $\Upsilon$ exists such that $L \Upsilon=   \Upsilon L=\Pi$  with  $\Pi:=I_{N}-\frac{1}{N} \mathbf{1}_{N} \mathbf{1}_{N}^{\top}$.
	\end{itemize}
\end{lem}


\begin{lem}\cite{Guo:2023Auto}\label{lem5}
	Given a connected and undirected graph $\mathcal{G}$, let its adjacency matrix and Laplacian matrix be denoted as $A$ and $L$, respectively. Then for $x_1,x_2,\ldots,x_N \in \mathbb{R}^{n}$ and $\boldsymbol{x}=[x_1^{\top},x_2^{\top},\cdots,x_N^{\top}]^{\top} \in \mathbb{R}^{nN}$ the following equations hold $\boldsymbol{x}^{\top}(L\otimes I_n)\boldsymbol{x}=\frac{1}{2}\sum_{i=1}^{N}\sum_{j=1}^{N}a_{ij}(x_i-x_j)^{\top}(x_i-x_j)$.
\end{lem}
\begin{lem}\cite{Wang:2020Auto}\label{lem6}
	Given a connected and undirected graph $\mathcal{G}$ and its corresponding complete graph $\overline{\mathcal{G}}$, with their Laplacian matrices denoted as $L$ and $L_{\overline{\mathcal{G}}}$, the following equation holds $\lambda_{2}(L)L_{\overline{\mathcal{G}}}\leq NL$.
\end{lem}

\subsection{Convex analysis}
\begin{lem}\cite{Chen:2016Auto}\label{lem3}
	If function $f(x)$ is twice continuously differentiable and $\gamma-$strongly convex, then the following equivalent conditions hold, i.e.,
	\begin{equation}\nonumber
		\begin{array}{l}
			f\left(x_{1}\right)-f\left(x_{2}\right)-\nabla^{\top} f\left(x_{2}\right)\left(x_{1}-x_{2}\right) \geqslant \frac{\gamma}{2}\left\|x_{1}-x_{2}\right\|^{2}, \\
			\left[\nabla f\left(x_{1}\right)-\nabla f\left(x_{2}\right)\right]^{\top}\left(x_{1}-x_{2}\right) \geqslant \gamma\left\|x_{1}-x_{2}\right\|^{2}, \\
			\nabla^{2} f\left(x_{1}\right) \geqslant \gamma I_{n}.
		\end{array}
	\end{equation}
	Furthermore, if there exists $\Gamma$ such that $\nabla^{2} f\left(x_{1}\right) \leq \Gamma I_{n}$, then the following equations hold.
	\begin{equation}\nonumber
		\begin{array}{l}
			f\left(x_{1}\right)-f\left(x_{2}\right)-\nabla^{\top} f\left(x_{2}\right)\left(x_{1}-x_{2}\right) \leqslant \frac{\Gamma}{2}\left\|x_{1}-x_{2}\right\|^{2},\\
			\left[\nabla f\left(x_{1}\right)-\nabla f\left(x_{2}\right)\right]^{\top}\left(x_{1}-x_{2}\right) \leqslant \Gamma\left\|x_{1}-x_{2}\right\|^{2}.
		\end{array}
	\end{equation}
\end{lem}

\subsection{Problem formulation}
Consider a multi-agent network composed of $N$ agents, each possessing computation and communication capabilities. The communication topology of the network is abstracted as $\mathcal{G}$. Each agent independently possesses an optimization variable for interaction with its neighbors and a private local objective function. The objective is to minimize the sum of these local objective functions, addressing the following optimization problem.
\begin{equation}\label{problem1}
	\begin{array}{l}
		\min\limits_{x \in \mathbb{R}^{n}} \tilde{f}(x)=\sum\limits_{i=1}^{N} f_{i}\left(x\right), 
	\end{array}
\end{equation}
which is equivalent to the following distributed form,
\begin{equation}\label{problem2}
	\begin{array}{ll}
		\min\limits_{\boldsymbol{x} \in \mathbb{R}^{n N}} & f(\boldsymbol{x})=\sum\limits_{i=1}^{N} f_{i}\left(x_{i}\right), \\
		\text { s.t } & \left(L \otimes I_{n}\right) \boldsymbol{x}=0,
	\end{array}
\end{equation}
where $f_i: \mathbb{R}^{n} \to \mathbb{R}$ is the local objective function of agent $i$, $x_i \in \mathbb{R}^{n}$ is the local optimization variable of agent $i$ and $\boldsymbol{x}=[x_1^{\top},x_2^{\top},\cdots,x_N^{\top}]^{\top} \in \mathbb{R}^{nN}$. Denote $\nabla f(\boldsymbol{x})=[\nabla f_{1}(x_1)^{\top},\ldots,\nabla f_{N}(x_N)^{\top}]^{\top}$, $\nabla^{2} f(\boldsymbol{x})=\left[\begin{smallmatrix}
	\nabla^{2} f_{1}\left(x_{1}\right) &   &  \\
	&\ddots &   \\
	&   & \nabla^{2} f_{N}\left(x_{N}\right)
\end{smallmatrix}\right]$ for further analysis.

\begin{assum}\label{Assumption1}
	Graph $\mathcal{G}$ is connected and undirected.
\end{assum}
\begin{assum}\label{Assumption2}
	For each agent $i$, its local objective function $f_i(x_i)$ is twice continuously differentiable, $\gamma_{i}-$strongly convex, $\psi_{i}-$smooth with $\gamma_{i},\psi_{i}>0$, and satisfies $\gamma_{i} I_{n} \leqslant \nabla^{2} f_{i}\left(x_{i}(t)\right) \leqslant \Gamma_{i} I_{n}$ with a bounded $\Gamma_{i}>0$. 
\end{assum}


\section{Main Results}\label{Section 3}

\subsection{Multi/Single-stage structure}
Currently, most prescribed-time distributed optimization algorithms based on time-varying scale functions require multi-stage structures to achieve global convergence. In this context, multi-stage implies that each subsequent stage relies solely on the results of the previous one, with each stage having independent update processes and utilizing a single time-varying scale function. The multi-stage and single-stage structures for prescribed-time distributed optimization algorithms, as visualized in Fig.~\ref{fig_model1} and Fig.~\ref{fig_model2}, can be modeled as follows.

\textbf{Multi-stage\,(MS):}
\begin{equation}\label{MSmodel}
	\begin{array}{lrl}
		&\text{Stage\,$1$:}\,\chi_{1} = &\hspace{-6pt} \Xi_{1}\big(\boldsymbol{x}, f(\boldsymbol{x}),\nabla  f(\boldsymbol{x}),\nabla^{2}  f(\boldsymbol{x}),\\
		&& \varrho_{1}(t;t_{0},T_{1})\big), \\
		&\text{Stage\,$2$:}\,\chi_{2} = &\hspace{-6pt} \Xi_{2}\big(\chi_{1}^{\ast},\boldsymbol{x}, f(\boldsymbol{x}),\nabla  f(\boldsymbol{x}),\nabla^{2}  f(\boldsymbol{x}),\\
		&& \varrho_{2}(t;t_{1},T_{2})\big), \\
		&& \cdots \\
		&\text{Stage\,$k$:}\,\chi_{k} = &\hspace{-6pt} \Xi_{k}\big(\chi_{1}^{\ast},\ldots,\chi_{k-1}^{\ast},\boldsymbol{x}, f(\boldsymbol{x}),\nabla  f(\boldsymbol{x}),\\
		&& \nabla^{2}  f(\boldsymbol{x}),\varrho_{k}(t;t_{k-1},T_{k})\big).
	\end{array}
\end{equation}

For stage $k'\,(1< k'\leq k)$ in this model, $\chi_{k'}=[\chi_{k'1}^{\top},\chi_{k'2}^{\top},\ldots,\chi_{k'm}^{\top}]^{\top}$ represents the $m$ different variables that need to be updated where the $i$th variable $\chi_{k'i}^{\top}=[(\chi_{k'i}^{(0)})^{\top},(\chi_{k'i}^{(1)})^{\top},(\chi_{k'i}^{(2)})^{\top}]^{\top}$ contains information of different orders of derivatives, and $\chi_{1}^{\ast},\ldots,\chi_{k'-1}^{\ast}$ are the optimal solutions obtained from the updates in the previous stages. $\varrho_{k'}(t;t_{k'-1},T_{k'})$ is a time-varying scaling function defined on the interval from $t_{k'-1}$ to $t_{k'-1}+T_{k'}$ and $\Xi_{k'}$ represents the dynamic input for updates.

\begin{rem}
	The aforementioned multi-stage model, due to its explicit stage divisions, can also be referred to as \textbf{explicit multi-stage (EMS)} model. The multi-stage model based distributed algorithms commonly involve stages such as `consensus stage', `global/local information estimation stage', and `global/local optimization stage'. Typically, the results from the previous stage are used to facilitate the subsequent stages. For example, the two-stage ZGS algorithm in \cite{Chen:2022IS} estimates the local minimizers in the first stage to satisfy the ZGS initial conditions before applying the ZGS algorithm in the second stage.
\end{rem}

\textbf{Single-stage\,(SS):}
\begin{equation}\label{SSmodel}
	\begin{array}{l}
		\dot{\chi} = \Xi\big(\boldsymbol{x}, f(\boldsymbol{x}),\nabla  f(\boldsymbol{x}),\nabla^{2}  f(\boldsymbol{x}),\varrho_{1}(t;t_{0},T_{1})\big).
	\end{array}
\end{equation}
Similarly, $\chi=[\chi_{1}^{\top},\chi_{2}^{\top},\ldots,\chi_{m}^{\top}]^{\top}$ denotes the $m$ different variables that need to be updated where the $i$th variable $\chi_{i}^{\top}=[(\chi_{i}^{(0)})^{\top},(\chi_{i}^{(1)})^{\top},(\chi_{i}^{(2)})^{\top}]^{\top}$ contains information of different orders of derivatives. At this point, there are no prior-stage results available for reference and use, as in \eqref{MSmodel}. $\varrho_{1}(t;t_{0},T_{1})$ is the only one time-varying scaling function used in the algorithm and $\Xi$ represents the dynamic input for updates.

\begin{figure}[htbp]
	\centerline{\includegraphics[width=0.45\textwidth]{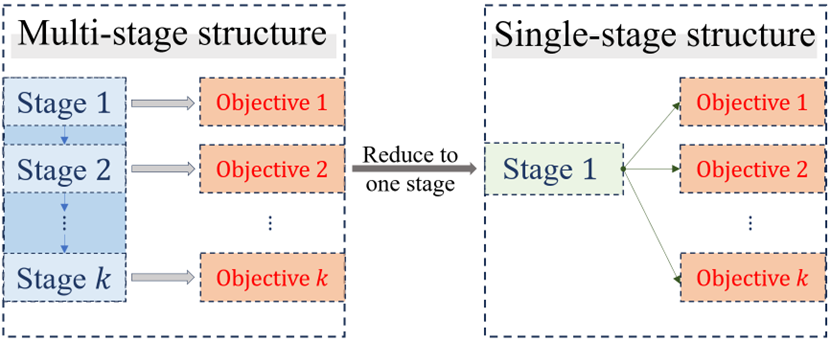}}
	\caption{The multi/single-stage structure.}
	\label{fig_model1}
\end{figure}

Comparing the two models \eqref{MSmodel} and \eqref{SSmodel} mentioned above, it can be observed that single-stage model algorithms are not only structurally more concise but also more convenient for practical use. However, multi-stage model algorithms have the advantage of directly integrating various optimization algorithms according to the objectives of each stage. In contrast, single-stage model algorithms face challenges in both design and analysis because they need to simultaneously achieve all the objectives of the multi-stage model within a single predefined time interval. As a result, most current prescribed-time distributed optimization algorithms based on time-varying scale functions are multi-stage.

\begin{figure}[htbp]
	\centerline{\includegraphics[width=0.45\textwidth]{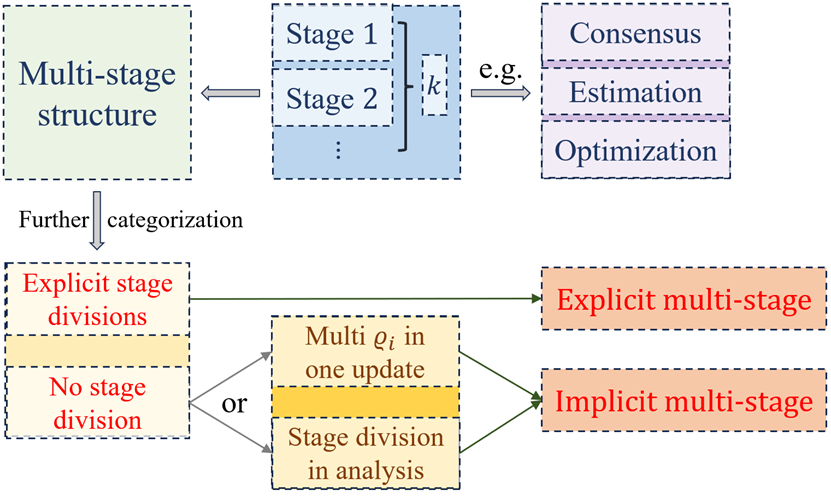}}
	\caption{The explicit/implicit multi-stage structure.}
	\label{fig_model2}
\end{figure}

\begin{rem}
	It is worth noting that even if some algorithms may appear to be structurally similar to single-stage algorithms, they should be considered \textbf{implicit multi-stage (IMS)} algorithms under the following circumstances, i.e.,
	\begin{itemize}
		\item [1)] the simultaneous use of multiple time-varying scaling functions in one update.
		\item [2)] the need for artificial temporal segmentation in the convergence analysis.
	\end{itemize}
	The above considerations also apply to finite-time convergence, fixed-time convergence, and other prescribed-time convergence frameworks. For example, the ZGS algorithm in \cite{Guo:2023Auto}, while not explicitly structured as multi-stage, still defines two time intervals in its fixed-time convergence analysis: one for satisfying the initial conditions of ZGS, and the other for ensuring global convergence. Furthermore, the latter stage solely relies on the results of the former stage, with non-intersecting updates, essentially making it a two-stage algorithm.
\end{rem}

\subsection{ZGS algorithm: research status}
The above models, specific to the ZGS algorithm, usually takes the following form.
\begin{equation}\label{MSmodel-ZGS}
	\begin{array}{rl}
		\dot{x}_{i}(t)=&\hspace{-6pt} -\left(\nabla^{2} f_{i}\left(x_{i}(t)\right)\right)^{-1}\Big[ \zeta_{1}\big(\varrho_{\iota_{1}}(t),y_{i}(t)  \big) \\
		&\hspace{-6pt}+ \zeta_{2}\big(\varrho_{\iota_{2}}(t),\sum\limits_{j \in \mathcal{N}_{i}}a_{i j}\left(x_{i}(t)-x_{j}(t)\right) \big) \Big],
	\end{array}
\end{equation}
where $\zeta_{1}$ and $\zeta_{2}$ denote protocols related to convergence control and consensus control, respectively. $\varrho_{\iota_{1}}(t)$ and $\varrho_{\iota_{2}}(t)$ are two time-varying scaling functions that can be either identical, signifying their equality in a single-stage algorithm, or distinct, signifying their inequality in a multi-stage algorithm (in this context, `equal' or `not equal' indicates whether they are defined over the same time interval).

Recall from \cite{Lu:2012TAC} that the traditional continuous-time ZGS has the following form,
\begin{equation}\label{ZGS}
	\begin{array}{rl}
		\dot{x}_{i}(t) = &\hspace{-6pt}\varphi_{i}\left(x_{i}(t), \mathbf{x}_{\mathcal{N}_{i}}(t) ; f_{i}, \mathbf{f}_{\mathcal{N}_{i}}\right), 
		\forall t \geq 0, \forall i \in \mathcal{V} \\
		x_{i}(t_{0}) = &\hspace{-6pt}\chi_{i}\left(f_{i}, \mathbf{f}_{\mathcal{N}_{i}}\right), \quad \forall i \in \mathcal{V}
	\end{array}
\end{equation}
where $\mathbf{x}_{\mathcal{N}_{i}}(t)$ represents the optimization variable information of node $i$'s neighbors, while $\mathbf{f}_{\mathcal{N}_{i}}$ represents the function value information of node $i$'s neighbors. $\varphi_{i}$ is the local dynamic input to be designed which usually required the following two stringent conditions,
\begin{subequations}\label{ZGScon}
	\begin{alignat}{1}
		&\sum\limits_{i \in \mathcal{V}} \nabla f_{i}\left(x_{i}(t_{0})\right)=0,\label{ZGScon1} \\
		&\sum_{i \in \mathcal{V}} \nabla^{2} f_{i}\left(x_{i}\right) \varphi_{i}\left(x_{i}, \mathbf{x}_{\mathcal{N}_{i}} ; f_{i}, \mathbf{f}_{\mathcal{N}_{i}}\right)=0.\label{ZGScon2} 
	\end{alignat}
\end{subequations}

The presence of these two conditions introduces certain inconveniences in both the utilization and construction of ZGS algorithms. While feasible solutions for \eqref{ZGScon1}, such as introducing an additional stage for local objective function minimization, are already available, there is still a lack of research on algorithms addressing \eqref{ZGScon2}.


\subsection{MS-PTZGS algorithm: multi-stage ZGS with prescribed-time performance}

To solve problem \eqref{problem1} without requiring \eqref{ZGScon1}, the following multi-stage PTDO algorithm using sliding mode control method is proposed,

\begin{subequations}\label{al1}
	\begin{alignat}{1}
		\dot{x}_{i}(t)=&\left(\nabla^{2} f_{i}\left(x_{i}(t)\right)\right)^{-1}\big[-(\kappa_{1}+  \frac{\dot{\varrho}_{1}(t)}{\varrho_{1}(t)})s_{i}(t)\notag\\
		&-c(\kappa_{2}\frac{\dot{\varrho}_{2}(t)}{\varrho_{2}(t)}) \sum\nolimits_{j=1}^{N} a_{i j}\left(x_{i}(t)-x_{j}(t)\right)\big],\label{al1-1} \\
		\dot{\phi}_{i}(t)=&\left(\kappa_{2} \frac{\dot{\varrho}_{2}(t)}{\varrho_{2}(t)}\right) \sum\nolimits_{j=1}^{N} a_{i j}\left(x_{i}(t)-x_{j}(t)\right), \label{al1-2}\\
		s_{i}(t)=&\nabla f_{i}\left(x_{i}(t)\right)+c \phi_{i}(t),\label{al1-3}
	\end{alignat}
\end{subequations}
where $c,\kappa_{1},\kappa_{2}>0$ are tuning parameters, $\varrho_{1}(t),\varrho_{2}(t)$ are the time-scaling functions defined in \eqref{mu1} and \eqref{mu2}, $s_{i}(t)$ is the distributed integral sliding manifold for agent $i$ to avoid additional local minimization to satisfy the requied initial condition $\sum_{i=1}^{N}\nabla f_{i}(x_{i}(t_{0}))=0$ for ZGS.

\begin{thm}\label{thm1}
	Under Assumptions \ref{Assumption1} and \ref{Assumption2}, for any initial value $x_i(t_{0})$, the proposed algorithm \eqref{al1} can solve the optimization problem \eqref{problem1} within a prescribed time interval $T_{1}+T_{2}$, i.e., $\lim\nolimits _{t \rightarrow\left(t_{0}+T_{1}+T_{2}\right)^{-}}x_{i}=x^{\ast}$ where $x^{\ast} \in \mathbb{R}^{n}$ is the global minimizer.
\end{thm}
\begin{proof}
	First, it is proven that the sliding manifold \eqref{al1-3} allows the satisfaction of the required initial condition $\sum_{i=1}^{N}\nabla f_{i}(x_{i}(t_{0}))=0$ within the time interval $T_{1}$ without the need to calculate the local minimizer $x^{\ast}_{i}$ of each $f_{i}(x)$.
	Define a Lyapunov function using $s_{i}$
	\begin{equation}\label{eq7}
		V_{i}(t)=\frac{1}{2}s_{i}^{\top}s_{i}.
	\end{equation}
	To ensure formula simplicity, the $(t)$ in $x_{i},s_{i}$ will be omitted. Differentiating $V_{i}(t)$ along \eqref{al1-3} yields
	\begin{equation}\label{eq8}
		\begin{array}{rl}
			\dot{V}_i(t) =&\hspace{-6pt} s_{i}^{\top}\dot{s}_{i} \\
			=&\hspace{-6pt} s_{i}^{\top}\big[ \nabla^{2}f_{i}(x_{i})\dot{x_{i}}\\
			&\hspace{-6pt}+c\kappa_{1} \frac{\dot{\varrho}_{2}(t)}{\varrho_{2}(t)} \sum\nolimits_{j=1}^{N} a_{i j}\left(x_{i}-x_{j}\right)\big] \\
			=&\hspace{-6pt} s_{i}^{\top}\big[ -(\kappa_{1}+\frac{\dot{\varrho}_{1}(t)}{\varrho_{1}(t)})s_{i}\\
			&\hspace{-6pt}-c\kappa_{2} \frac{\dot{\varrho}_{2}(t)}{\varrho_{2}(t)} \sum\nolimits_{j=1}^{N} a_{i j}\left(x_{i}-x_{j}\right) \\
			&\hspace{-6pt}+c\kappa_{2} \frac{\dot{\varrho}_{2}(t)}{\varrho_{2}(t)} \sum\nolimits_{j=1}^{N} a_{i j}\left(x_{i}-x_{j}\right)\big] \\
			=&\hspace{-6pt} -(\kappa_{1}+\frac{\dot{\varrho}_{1}(t)}{\varrho_{1}(t)})s_{i}^{\top}s_{i}\\
			=&\hspace{-6pt} -(2\kappa_{1}+2\frac{\dot{\varrho}_{1}(t)}{\varrho_{1}(t)})V_{i}(t).
		\end{array}
	\end{equation}
	Using Lemma \ref{lem7}, there is $\lim\nolimits _{t \rightarrow\left(t_{0}+T_{1}\right)^{-}} V_{i}(t)=0$ and $V(x, t) \equiv 0, \forall t \in\left[t_{0}+T_{1}, \infty\right)$, i.e., $\lim\nolimits _{t \rightarrow\left(t_{0}+T_{1}\right)^{-}} s_{i}=0$ and $s_{i} \equiv 0, \forall t \in\left[t_{0}+T_{1}, \infty\right)$. Then when $t\in [t_{0}+T_{1},+\infty)$ there is $s_{i}=0$. 
	
	Furthermore, since 
	\begin{equation}
		\begin{array}{l}
			\sum\nolimits_{i=1}^{N}c\phi_{i}(t) \\
			=\sum\nolimits_{i=1}^{N}c \int_{t_{0}}^{t}\left(\kappa_{2} \frac{\dot{\varrho}_{2}(\tau)}{\varrho_{2}(\tau)}\right) \sum\nolimits_{j=1}^{N} a_{i j}\left(x_{i}(\tau)-x_{j}(\tau)\right) \mathrm{d} \tau \\
			=0,
		\end{array}
	\end{equation}
	it is easy to obtain that
	$\sum_{i=1}^{N}s_{i}(t_{1})=\sum_{i=1}^{N}\nabla f_{i}(x_{i}(t_{1}))$. Now it is proven that the condition $\sum_{i=1}^{N}\nabla f_{i}(x_{i})=0$ can be achieved within time interval $T_{1}$ using \eqref{al1-3}.
	
	Next proof that all local state variables $x_{i},i=1,\ldots,N$ will converge to $x^{\ast}$ within the next time interval $[t_{1},t_{1}+T_{2})$. From the analysis form above, since $s_{i} \equiv 0, \forall t \in\left[t_{0}+T_{1}, \infty\right)$, \eqref{al1-1} is equivalent to
	\begin{equation}\label{al2-1} 
		\begin{array}{rl}
			\dot{x}_{i}=&\hspace{-6pt}\left(\nabla^{2} f_{i}\left(x_{i}(t)\right)\right)^{-1}\big[-c(\kappa_{2}\frac{\dot{\varrho}_{2}(t)}{\varrho_{2}(t)})\\
			 &\hspace{-6pt}\sum\nolimits_{j=1}^{N} a_{i j}\left(x_{i}-x_{j}\right)\big].
		\end{array}
	\end{equation}
	Define another Lyapunov function as
\begin{equation}\label{eq10}
	\begin{array}{l}
		V_{M}(t)=\sum\nolimits_{i=1}^{N}\big[f_{i}\left(x^{*}\right)-f_{i}\left(x_{i}\right)-\nabla^{\top} f_{i}\left(x_{i}\right)\left(x^{*}-x_{i}\right)\big].	
	\end{array}
\end{equation}
	Differentiating $V_{M}(t)$ along \eqref{al2-1} yields
	\begin{equation}\label{eq11}
		\begin{array}{rl}
			\dot{V}_{M}(t) =&\hspace{-6pt} \sum\nolimits_{i=1}^{N}\left[-\nabla^{2} f_{i}\left(x_{i}\right) \dot{x}_i \left(x^{*}-x_{i}\right) \right] \\
			=&\hspace{-6pt} c\kappa_{2}\frac{\dot{\varrho}_{2}(t)}{\varrho_{2}(t)} \sum\nolimits_{i=1}^{N}\sum\nolimits_{j=1}^{N} a_{i j}\left(x^{*}-x_{i}\right)^{\top}\left(x_{i}-x_{j}\right) .			
		\end{array}
	\end{equation}
	Using Lemma \ref{lem5}, there is
	\begin{equation}\label{eq12}
		\begin{array}{rl}
			\dot{V}_{M}(t) =&\hspace{-6pt} -\frac{1}{2}c\kappa_{2}\frac{\dot{\varrho}_{2}(t)}{\varrho_{2}(t)} \sum\nolimits_{i=1}^{N}\sum\nolimits_{j=1}^{N} a_{i j}\left(x_{i}-x_{j}\right)^{\top}\left(x_{i}-x_{j}\right)	\\
			=&\hspace{-6pt} -c\kappa_{2}\frac{\dot{\varrho}_{2}(t)}{\varrho_{2}(t)}\boldsymbol{x}^{\top}(L\otimes I_{n}) \boldsymbol{x}.		
		\end{array}
	\end{equation}
	Define $\overline{x}=\frac{1}{N}\sum_{i=1}^{N}x_{i}$. From $\sum_{i=1}^{N}\nabla f_{i}(x_{i})=0$ and $\tilde{f}(\overline{x}) \ge \tilde{f}(x^{\ast})$, it can be obtained that
	\begin{equation}\label{eq13}
		\begin{array}{l}
			\sum\nolimits_{i=1}^{N}\left[f_{i}\left(\overline{x}\right)-f_{i}\left(x_{i}\right)-\nabla^{\top} f_{i}\left(x_{i}\right)\left(\overline{x}-x_{i}\right)\right]-V_{M}(t) \\
			= \sum\nolimits_{i=1}^{N}\left[ f_{i}\left(\overline{x}\right)-f_{i}\left(x^{\ast}\right)- \nabla^{\top} f_{i}\left(x_{i}\right)(\overline{x}-x^{\ast})\right] \\
			= \tilde{f}(\overline{x}) - \tilde{f}(x^{\ast}) - (\overline{x}-x^{\ast})^{\top}\sum\nolimits_{i=1}^{N}\nabla f_{i}\left(x_{i}\right) \\
			= \tilde{f}(\overline{x}) - \tilde{f}(x^{\ast}) \ge 0.
		\end{array}
	\end{equation}
	Then there is 
	\begin{equation}\label{eqtemp1}
		V_{M}(t) \leq \sum_{i=1}^{N}\left[f_{i}\left(\overline{x}\right)-f_{i}\left(x_{i}\right)-\nabla^{\top} f_{i}\left(x_{i}\right)\left(\overline{x}-x_{i}\right)\right].
	\end{equation}
	By Assumption \ref{Assumption2} and Lemma \ref{lem3}, it can be derived that $f_{i}\left(\overline{x}\right)-f_{i}\left(x_{i}\right)-\nabla^{\top} f_{i}\left(x_{i}\right)\left(\overline{x}-x_{i}\right) \leq \frac{\Gamma_{i}}{2}\left\|x_{i}-\overline{x}\right\|^{2}$. Taking $\Gamma_{\max}$ as $\max _{i \in \mathcal{V}}\left\{\Gamma_{i}\right\}$, there is $V_{M}(t) \leq \frac{\Gamma_{\max}}{2}\sum_{i=1}^{N}\left\|x_{i}-\overline{x}\right\|^{2}$. Using Lemma \ref{lem5} and \ref{lem6}, it can be derived that
	\begin{equation}\label{eq14}
		\begin{array}{rl}
			V_{M}(t) \leq &\hspace{-6pt} \frac{\Gamma_{\max}}{2}\sum\nolimits_{i=1}^{N}\left\|x_{i}-\overline{x}\right\|^{2}	\\
			=&\hspace{-6pt} \frac{\Gamma_{\max}}{2N^{2}}\sum\nolimits_{i=1}^{N}\left\|\sum\nolimits_{j=1}^{N}(x_{i}-x_{j})\right\|^{2} \\
			\leq &\hspace{-6pt} \frac{\Gamma_{\max}}{2N^{2}}N\sum\nolimits_{i=1}^{N}\sum\nolimits_{j=1}^{N}\left\|(x_{i}-x_{j})\right\|^{2} \\
			=&\hspace{-6pt} \frac{\Gamma_{\max}}{N}\boldsymbol{x}^{\top}(L_{\overline{\mathcal{G}}} \otimes I_{n})\boldsymbol{x} \\
			\leq &\hspace{-6pt} \frac{\Gamma_{\max}}{\lambda_{2}(L)}\boldsymbol{x}^{\top}(L \otimes I_{n})\boldsymbol{x}.
		\end{array}
	\end{equation}
	Combining \eqref{eq12} and \eqref{eq14}, there is
	\begin{equation}\label{eq15}
		\dot{V}_{M}(t) \leq -\frac{\lambda_{2}(L)c\kappa_{2}}{\Gamma_{\max}}\frac{\dot{\varrho}_{2}(t)}{\varrho_{2}(t)}V_{M}(t).
	\end{equation}
	Letting $\alpha_{1}=\frac{\lambda_{2}(L)c\kappa_{2}}{\Gamma_{\max}}$ and multiplying $\varrho_{2}^{\alpha_{1}}(t)$ on both sides of \eqref{eq15}, it can be obtained that $\varrho_{2}^{\alpha_{1}}(t)\dot{V}_{M}(t) \leq -\alpha_{1} \dot{\varrho}_{2}(t)\varrho_{2}^{\alpha_{1}-1}(t)V_{M}(t)$. Then there is $\frac{d(\varrho_{2}^{\alpha_{1}}(t)V_{M}(t))}{dt} \leq 0$ which can be derived as
	\begin{equation}
		V_{M}(t) \leq V_{M}(t_{0})\varrho_{2}^{-\alpha_{1}}(t),
	\end{equation}
	which satisfies the conditions \eqref{PTC} in Definition \ref{def1}. Then there is $\lim\nolimits _{t \rightarrow\left(t_{1}+T_{2}\right)^{-}} V_{M}(t)=0$ and $V_{M}(t) \equiv 0, \forall t \in\left[t_{1}+T_{2}, \infty\right)$, i.e., $\lim\nolimits _{t \rightarrow\left(t_{1}+T_{2}\right)^{-}} x_{i}(t)=x^{\ast}$ and $x_{i}(t) \equiv x^{\ast}, \forall t \in\left[t_{1}+T_{2}, \infty\right)$. The proof is now complete.
\end{proof}
\begin{rem}
	In summary, the proposed MS-PTZGS demonstrates excellent prescribed-time convergence performance. Compared to \cite{Chen:2022IS}, this algorithm does not require condition \eqref{ZGScon1}, and compared to LMFZGS in \cite{Guo:2023Auto}, it achieves prescribed-time convergence performance far exceeding fixed-time convergence performance. However, this algorithm cannot guarantee the absence of condition \eqref{ZGScon2} in all stages, only in the first stage. The upcoming single-stage algorithm, to be introduced shortly, will address this issue.
\end{rem}
\begin{rem}
	The MS-PTZGS algorithm proposed here belongs to the implicit multi-stage algorithm mentioned earlier, which still maintains the simplicity of a similar single-stage algorithm (i.e., it does not have to run each stage individually) in terms of algorithmic operation as compared to the explicit multi-stage algorithm.
\end{rem}
\subsection{SS-PTZGS algorithm: single-stage ZGS with prescribed-time performance}
To solve problem \eqref{problem1} without requiring both \eqref{ZGScon1} and \eqref{ZGScon2}, the following single-stage PTDO algorithm using sliding mode control method is proposed,
\begin{subequations}\label{al2}
	\begin{alignat}{1}
		\dot{x}_{i}(t)=&\left(\nabla^{2} f_{i}\left(x_{i}(t)\right)\right)^{-1}\big(\kappa_{1} \frac{\dot{\varrho}_{1}(t)}{\varrho_{1}(t)}\big)\big(-\kappa_{2} s_{i}(t)\notag\\
		& -c \sum\nolimits_{j=1}^{N} a_{i j}\left(x_{i}(t)-x_{j}(t)\right)\big),\label{al2_1} \\
		s_{i}(t)=&c \int_{0}^{t}\big(\kappa_{1} \frac{\dot{\varrho_{1}}(\tau)}{\varrho_{1}(\tau)}\big) \sum\nolimits_{j=1}^{N} a_{i j}\left(x_{i}(\tau)-x_{j}(\tau)\right) \mathrm{d} \tau \notag \\
		& +\nabla f_{i}\left(x_{i}(t)\right),\label{al2_2}
	\end{alignat}
\end{subequations}
where $c,\kappa_{1},\kappa_{2}>0$ are tuning parameters, $\varrho_{1}(t)$ is the only used time-scaling function defined in \eqref{mu1}, $s_{i}(t)$ is the distributed integral sliding manifold for agent $i$ to avoid additional local minimization to satisfy the requied initial condition \eqref{ZGScon1}. Denote $\boldsymbol{s}=[s_1^{\top},s_2^{\top},\cdots,s_N^{\top}]^{\top} \in \mathbb{R}^{nN}$ and $\boldsymbol{x}^{\ast}=\mathbf{1}_{N}\otimes x^{\ast} \in \mathbb{R}^{nN}$.

\begin{rem}
	There are two main motivations behind the construction of this single-stage structured PTDO, i.e.,
	\begin{itemize}
		\item [1)] currently, multi-stage non-prescribed-time convergence algorithms have the drawback of significant disparity between theoretical convergence time and practical convergence time (usually the former being larger). This inability to accurately estimate convergence time is not conducive to practical engineering applications.
		\item [2)] in the existing literature related to ZGS algorithm, articles either exhibit a multi-stage structure in algorithm design or convergence proofs. This is primarily due to the inherent difficulty in simultaneously achieving the goals of gradual gradient sum reduction to $0$ and achieving global convergence in a single stage.
	\end{itemize}
\end{rem}

\begin{table*}
	\centering
	\caption{Comparison with existing ZGS algorithms}
	\label{tb1}
	\begin{tabular}{|cc|c|cc|c|cc|}
		\hline
		\multicolumn{2}{|c|}{\multirow{2}{*}{\textbf{Type}}}                           & \multirow{2}{*}{\textbf{Algorithm}} & \multicolumn{2}{c|}{\textbf{Protocals}}        & \multirow{2}{*}{\textbf{Convergence rate}} & \multicolumn{2}{c|}{\textbf{Is it not bound by}}                                         \\ \cline{4-5} \cline{7-8} 
		\multicolumn{2}{|c|}{}                                                         &                                     & \multicolumn{1}{c|}{$\zeta_{1}$} & $\zeta_{2}$ &                                            & \multicolumn{1}{c|}{\eqref{ZGScon1}} & \eqref{ZGScon2} \\ \hline
		\multicolumn{1}{|c|}{\multirow{6}{*}{Multi-stage}} & \multirow{3}{*}{Implicit} & MS-PTZGS                            & \multicolumn{1}{c|}{$(\kappa_{1}+  \frac{\dot{\varrho}_{1}}{\varrho_{1}})s_{i}$}            & $c(\kappa_{2}\frac{\dot{\varrho}_{2}}{\varrho_{2}}) \sum\limits_{j=1}^{N} a_{i j}\left(x_{i}-x_{j}\right)$            & prescribed-time                            & \multicolumn{1}{c|}{yes}                              & no                               \\ \cline{3-8} 
		\multicolumn{1}{|c|}{}                             &                           & Algorithm in \cite{Shi:2023CSL}                         & \multicolumn{1}{c|}{$\varphi_{i}\left(z_{i}\right)+\frac{\partial \nabla f_{i}\left(x_{i}, t\right)}{\partial t}$}            & $\alpha\sum\limits_{j=1}^{N} a_{i j} \operatorname{sgn}\left(x_{i}-x_{j} \right)$            & finite-time                                & \multicolumn{1}{c|}{yes}                              & yes                              \\ \cline{3-8} 
		\multicolumn{1}{|c|}{}                             &                           & Algorithm in \cite{Guo:2023Auto}                        & \multicolumn{1}{c|}{$\begin{array}{rl}
				&\hspace{-6pt}k_{1} \operatorname{sig}^{p_{1}}\left(s_{i}\right)\\+&\hspace{-6pt}k_{2} \operatorname{sig}^{q_{1}}\left(s_{i}\right)
			\end{array}$}            & $\begin{array}{rl}
			&\hspace{-6pt}c_{1} \sum\limits_{j=1}^{N} a_{i j} \operatorname{sig}^{p_{2}}\left(x_{i}-x_{j}\right)\\+&\hspace{-6pt}c_{2} \sum\limits_{j=1}^{N} a_{i j} \operatorname{sig}^{q_{2}}\left(x_{i}-x_{j}\right)
		\end{array}$            & fixed-time                                 & \multicolumn{1}{c|}{yes}                              & no                               \\ \cline{2-8} 
		\multicolumn{1}{|c|}{}                             & \multirow{3}{*}{Explicit} & Algorithm in \cite{Chen:2022IS}                         & \multicolumn{1}{c|}{$\left(k_{1} \frac{\dot{\varrho}_{1}}{\varrho_{1}}\right) y_{i}(t_{\sigma}^{i})$}            & $\begin{array}{rl}
			&\hspace{-6pt}(k_{2} \frac{\dot{\varrho}_{2}}{\varrho_{2}})\big(\sum\limits_{j=1}^{N} a_{i j}(x_{i}\left(t_{k}^{i}\right)\\&\hspace{-6pt}-x_{j}(t_{k^{\prime}}^{j}))\big)
		\end{array}$            & prescribed-time                            & \multicolumn{1}{c|}{yes}                              & no                               \\ \cline{3-8} 
		\multicolumn{1}{|c|}{}                             &                           & Algorithm in \cite{Chen:2023TCNS}                       & \multicolumn{1}{c|}{$\dot{\omega}_{1}\nabla f_{i}\left(x_{i}(t_{0})\right)$}            & $\begin{array}{l}
			(\frac{\dot{\omega}_{2}}{1-\omega_{2}+\beta}+1)\times  \\
			\sum\limits_{j=1}^{N} a_{i j}\left(\hat{x}_{j}^{i}-\hat{x}_{i}^{j}\right)
		\end{array}$            & prescribed-time                            & \multicolumn{1}{c|}{yes}                              & no                               \\ \cline{3-8} 
		\hline \multicolumn{2}{|c|}{\multirow{6}{*}{Single-stage}}                            & SS-PTZGS                            & \multicolumn{1}{c|}{$ \kappa_{1}\kappa_{2} \frac{\dot{\varrho}_{1}}{\varrho_{1}} s_{i}$}            & $c(\kappa_{1}\frac{\dot{\varrho}_{1}}{\varrho_{1}}) \sum\limits_{j=1}^{N} a_{i j}\left(x_{i}-x_{j}\right)$            & prescribed-time                            & \multicolumn{1}{c|}{yes}                              & yes                              \\ \cline{3-8} 
		\multicolumn{2}{|c|}{}                                                         & Algorithm in \cite{Wu:2022SMCS}                         & \multicolumn{1}{c|}{not designed}            & $\begin{array}{rl}
			&\hspace{-6pt}\alpha_{1} \sum\limits_{j=1}^{N} a_{i j} \operatorname{sig}^{p}\left(x_{i}-x_{j}\right) \\
			+&\hspace{-6pt}\alpha_{2}\sum\limits_{j=1}^{N} a_{i j} \operatorname{sig}^{q}\left(x_{i}-x_{j}\right)
		\end{array}$            & finite-time                                & \multicolumn{1}{c|}{no}                               & no                               \\ \cline{3-8} 
		\multicolumn{2}{|c|}{}                                                         & Algorithm in \cite{Zhao:2021TF}                         & \multicolumn{1}{c|}{not designed}            & $\sum\limits_{j=1}^{N}a_{i j}\left(\hat{x}_{j}(t-\tau)-\hat{x}_{i}(t-\tau)\right)$            & not guaranteed                             & \multicolumn{1}{c|}{no}                               & no                               \\ \cline{3-8} 
		\multicolumn{2}{|c|}{}                                                         & Algorithm in \cite{Chen:2016Auto}                       & \multicolumn{1}{c|}{not designed}            & $\gamma\sum\limits_{j=1}^{N}a_{i j}\left(\hat{x}_{j}-\hat{x}_{i}\right)$            & not guaranteed                             & \multicolumn{1}{c|}{no}                               & no                               \\ \cline{3-8} 
		\multicolumn{2}{|c|}{}                                                         & Algorithm in \cite{Chen:2023ISAS}                       & \multicolumn{1}{c|}{not designed}            & $\sum\limits_{j=1}^{N}a_{i j}\left(\theta_{i}+\theta_{j}\right)\left(x_{j}-x_{i}\right)^{-1}$            & fixed-time                                 & \multicolumn{1}{c|}{no}                               & yes                              \\ \cline{3-8} 
		\multicolumn{2}{|c|}{}                                                         & Algorithm in \cite{Yu:2021ACC}                          & \multicolumn{1}{c|}{$c_{2}y_{i}$}            & $c_{1} \sum\limits_{j=1}^{N} a_{i j}\left(x_{i}-x_{j}\right)$            & Exponential                                & \multicolumn{1}{c|}{yes}                              & yes                              \\ \hline
	\end{tabular}
\end{table*}

\begin{thm}\label{thm2}
	Under Assumptions \ref{Assumption1} and \ref{Assumption2}, for any initial value $x_i(t_{0})$, the proposed algorithm \eqref{al2} can solve the optimization problem \eqref{problem1} within a single prescribed time interval $T_{1}$, i.e., $\lim\nolimits _{t \rightarrow\left(t_{0}+T_{1}\right)^{-}}x_{i}=x^{\ast}$ where $x^{\ast} \in \mathbb{R}^{n}$ is the global minimizer.
\end{thm}
\begin{proof}
		Similarly, to ensure formula simplicity, the $(t)$ in $x_{i},s_{i}$ will be omitted. First, it can be obtained from \eqref{al2} for $s_{i}$ that
	\begin{equation}\label{eq24}
		\begin{array}{rl}
			\dot{s}_{i}=&\hspace{-6pt} \nabla^{2}f_{i}(x_{i})\dot{x}_{i}+c\left(\kappa_{1} \frac{\dot{\varrho_{1}}(t)}{\varrho_{1}(t)}\right) \sum\nolimits_{j=1}^{N} a_{i j}\left(x_{i}-x_{j}\right)\\
			=&\hspace{-6pt} -\left(\kappa_{1} \frac{\dot{\varrho}_{1}(t)}{\varrho_{1}(t)}\right)\big(\kappa_{2} s_{i}+c \sum\nolimits_{j=1}^{N} a_{i j}\left(x_{i}-x_{j}\right)\big)\\
			&\hspace{-6pt}+c\left(\kappa_{1} \frac{\dot{\varrho_{1}}(t)}{\varrho_{1}(t)}\right) \sum\nolimits_{j=1}^{N} a_{i j}\left(x_{i}-x_{j}\right)\\
			=&\hspace{-6pt} -\left(\kappa_{1} \frac{\dot{\varrho}_{1}(t)}{\varrho_{1}(t)}\right)\kappa_{2} s_{i}.
		\end{array}
	\end{equation}
	Define another Lyapunov function as
	\begin{equation}\label{eq25}
		\begin{array}{rl}
			V_{S}(t)=&\hspace{-6pt} \frac{1}{2}p\boldsymbol{s}^{\top}\boldsymbol{s}+\sum\nolimits_{i=1}^{N}\big[f_{i}\left(x^{*}\right)-f_{i}\left(x_{i}\right)\\
			&\hspace{-6pt}-\nabla^{\top} f_{i}\left(x_{i}\right)\left(x^{*}-x_{i}\right)\big]\\
		\end{array}
	\end{equation}
	Taking $\Psi_{\max}$ as $\max\limits_{i \in \mathcal{V}}\left\{\psi_{i}\right\}$, from Assumption \ref{Assumption2}, there is
	\begin{equation}\label{eq26}
		\begin{array}{rl}
			V_{S}(t)\leq&\hspace{-6pt} \frac{\Psi_{\max}}{2}\left\|\boldsymbol{x}-\boldsymbol{x}^{\ast}\right\|^{2}+\frac{p}{2}\left\|\boldsymbol{s}\right\|^{2}.	
		\end{array}
	\end{equation}
	Differentiating $V_{S}(t)$ along \eqref{al2} yields
	\begin{equation}\label{eq27}
		\begin{array}{rl}
			\dot{V}_{S}(t) =&\hspace{-6pt} \sum\nolimits_{i=1}^{N}\left[-\nabla^{2} f_{i}\left(x_{i}\right) \dot{x}_i \left(x^{*}-x_{i}\right) \right]+p\boldsymbol{s}^{\top}\dot{\boldsymbol{s}} \\
			=&\hspace{-6pt} \kappa_{1}\frac{\dot{\varrho}_{1}(t)}{\varrho_{1}(t)} \sum\nolimits_{i=1}^{N}\left(x^{*}-x_{i}\right)^{\top}\big[c\sum\nolimits_{j=1}^{N} a_{i j}\left(x_{i}-x_{j}\right)\\
			&\hspace{-6pt}+ \kappa_{2}s_{i}\big]-\kappa_{1}\frac{\dot{\varrho}_{1}(t)}{\varrho_{1}(t)}p\kappa_{2}\boldsymbol{s}^{\top}\boldsymbol{s}\\
			=&\hspace{-6pt} -\kappa_{1}\frac{\dot{\varrho}_{1}(t)}{\varrho_{1}(t)}\big[c(\boldsymbol{x}-\boldsymbol{x}^{\ast})^{\top}(L\otimes I_{n})(\boldsymbol{x}-\boldsymbol{x}^{\ast})\\
			&\hspace{-6pt} +\kappa_{2}\boldsymbol{s}^{\top}(\boldsymbol{x}-\boldsymbol{x}^{\ast})\big]-\kappa_{1}\frac{\dot{\varrho}_{1}(t)}{\varrho_{1}(t)}p\kappa_{2}\boldsymbol{s}^{\top}\boldsymbol{s}.
		\end{array}
	\end{equation}
	Using Young's inequality, there exists $\delta >0$ such that 
	\begin{equation}\label{eq28}
		\begin{array}{rl}
			\dot{V}_{S}(t) \leq &\hspace{-6pt}	\kappa_{1}\frac{\dot{\varrho}_{1}(t)}{\varrho_{1}(t)}\big[-c\lambda_{2}(L)\left\|\boldsymbol{x}-\boldsymbol{x}^{\ast}\right\|^{2}-p\kappa_{2}\boldsymbol{s}^{\top}\boldsymbol{s} \\
			&\hspace{-6pt}+	\kappa_{2}\delta \boldsymbol{s}^{\top}\boldsymbol{s}	+ \frac{\kappa_{2}}{4\delta}\left\|\boldsymbol{x}-\boldsymbol{x}^{\ast}\right\|^{2} \big]\\
			=&\hspace{-6pt} -\kappa_{1}\frac{\dot{\varrho}_{1}(t)}{\varrho_{1}(t)}\big[ (c\lambda_{2}(L)-\frac{\kappa_{2}}{4\delta})\left\|\boldsymbol{x}-\boldsymbol{x}^{\ast}\right\|^{2}\\
			&\hspace{-6pt}+\kappa_{2}(p-\delta)\boldsymbol{s}^{\top}\boldsymbol{s} \big].
		\end{array}
	\end{equation}
	Assign values to $p$ and $\delta$ that satisfy the following condition,
	\begin{equation}\label{eq29}
		p>\delta>\frac{\kappa_{2}}{4c\lambda_{2}(L)}.
	\end{equation}
	Taking $\sigma_{S}=\min\left\{\frac{2c\lambda_{2}(L)}{\Psi_{\max}}-\frac{\kappa_{2}}{2\delta\Psi_{\max}},\frac{2\kappa_{2}(p-\delta)}{p}\right\}$ and recalling \eqref{eq26}, there is
	\begin{equation}\label{eq30}
		\dot{V}_{S}(t) \leq -\kappa_{1}\frac{\dot{\varrho}_{1}(t)}{\varrho_{1}(t)}\sigma_{S}V_{S}(t).
	\end{equation}
	Letting $\alpha_{2}=\kappa_{1}\sigma_{S}$ and multiplying $\varrho_{1}^{\alpha_{2}}(t)$ on both sides of \eqref{eq30}, it can be obtained that $\varrho_{1}^{\alpha_{2}}(t)\dot{V}_{S}(t) \leq -\alpha_{2} \dot{\varrho}_{1}(t)\varrho_{1}^{\alpha_{2}-1}(t)V_{S}(t)$. Then there is $\frac{d(\varrho_{1}^{\alpha_{2}}(t)V_{S}(t))}{dt} \leq 0$ which can be derived as
	\begin{equation}\label{eq31}
		V_{S}(t) \leq V_{S}(t_{0})\varrho_{1}^{-\alpha_{2}}(t),
	\end{equation}
	which satisfies the conditions \eqref{PTC} in Definition \ref{def1}. Then there is $\lim\nolimits _{t \rightarrow\left(t_{0}+T_{1}\right)^{-}} V_{S}(t)=0$ and $V_{S}(t) \equiv 0, \forall t \in\left[t_{0}+T_{1}, \infty\right)$, i.e., $\lim\nolimits _{t \rightarrow\left(t_{0}+T_{1}\right)^{-}} x_{i}(t)=x^{\ast}$ and $x_{i}(t) \equiv x^{\ast}, \forall t \in\left[t_{0}+T_{1}, \infty\right)$. The proof is now complete.
\end{proof}

\begin{rem}
	From the analysis above, it can be observed that the proposed SS-PTZGS simultaneously avoids the requirements of conditions \eqref{ZGScon1} and \eqref{ZGScon2} while achieving the objectives of making the gradient sum approach zero and achieving global convergence within a single prescribed-time interval. The approach to achieving this result is not overly complex. It involves adding the quadratic term of a sliding mode surface to the traditional Lyapunov function that guarantees global convergence. This ensures that each local variable, driven by the algorithm, gradually approaches both the sliding mode surface, ensuring the gradient sum becomes zero, and the global optimal solution.
\end{rem}

So far, the two algorithms proposed in this paper have both undergone a comprehensive analysis. To further illustrate and compare their distinctions and advantages compared to others' work, Table~\ref{tb1} summarizes the recent algorithmic developments related to the problems addressed, which also include the two algorithms proposed later in this paper.

\section{Simulation}\label{Section 4}
This section will verify the excellent prescribed-time convergence performance of the two proposed algorithms through a simulation. Consider an undirected graph with $N=6$ agents as shown in Fig.~\ref{fig1} and the six local objective functions for each agent are 
\begin{equation}\nonumber
	\begin{aligned}
		&f_1(x) &&\hspace{-10pt}= (x_{1}-1)^2+(x_{2}-2)^2, \, &&f_4(x) &&\hspace{-10pt}= x_{1}^2+2x_{2}^2, \\
		&f_2(x) &&\hspace{-10pt}= (x_{1}-3)^2+(x_{2}-4)^2, \, &&f_5(x) &&\hspace{-10pt}= 2x_{1}^2+x_{2}^2, \\
		&f_3(x) &&\hspace{-10pt}= (x_{1}-5)^2+(x_{2}-6)^2, \, &&f_6(x) &&\hspace{-10pt}= 3x_{1}^2+2x_{2}^2.
	\end{aligned}
\end{equation}
The selection of parameters for \eqref{al1} and \eqref{al2} is as follows.
\begin{table}[H]
	\centering
	\begin{tabular}{|c|c|c|c|c|c|c|c|}
		\hline
		Algorithm & $\kappa_{1}$ & $\kappa_{2}$ & $c$ & $T_{1}$ & $h_{1}$ & $T_{2}$ & $h_{2}$\\ \hline
		MS-PTZGS \eqref{al1}                  & 2           & 3           & 1          & 0.1s        & 3           & 0.2s        & 2.5         \\ \hline
		SS-PTZGS \eqref{al2}                 & 2           & 3           & 1          & 0.3s        & 2.3         &  \diagbox{}{}      & \diagbox{}{}          \\ \hline
	\end{tabular}
\end{table}
\begin{figure}[htbp]
	\centerline{\includegraphics[width=0.3\textwidth]{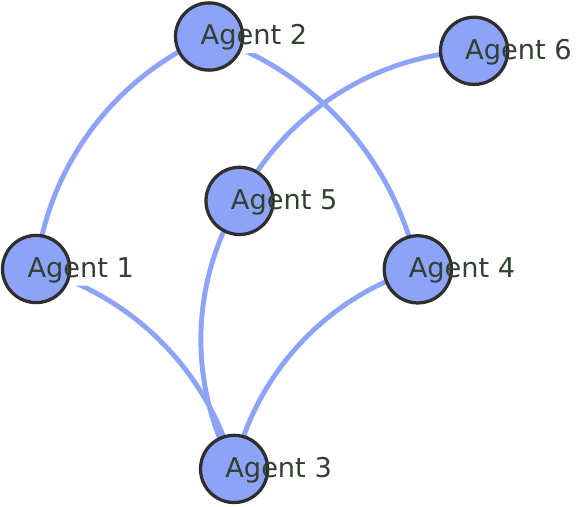}}
	\caption{Communication topology.}
	\label{fig1}
\end{figure}

Fig.~\ref{fig_xi} illustrates the temporal variations of individual components $x_{ij}$ of the local optimization variables $x_{i}$ for the two proposed algorithms. Fig.~\ref{fig_Re} further demonstrates the changes in residuals between each local optimization variable and the global optimum, where the residual is denoted by $\mathrm{Er}(x_{i}(t))=\frac{\left\|x_{i}(t)-x^{*}\right\|_{2}^{2}}{\left\|x_i(t_{0})-x^{*}\right\|_{2}^{2}}$. Subsequently, Fig.~\ref{fig_si} displays the temporal trends of the sliding mode surface term $s_{i}$, confirming the effectiveness of eliminating the requirement of \eqref{ZGScon1} using sliding mode control. Finally, Fig.~\ref{fig_y} depicts the evolution of the error between the global function value and the optimal value. From these four figures, the following key points can be summarized as follows,
\begin{itemize}
	\item [1)] both of the proposed algorithms successfully achieve global convergence within the prescribed time interval, and the small error between theoretical and actual convergence times is advantageous for meeting the requirements of real-world engineering applications that necessitate precise control over convergence time.
	\item [2)] MS-PTZGS, with its two-stage structure, exhibits a clear two-stage division in the simulation results, while the single-stage SS-PTZGS shows only one-stage convergence.
	\item [3)] the applied sliding mode surface term accomplishes the desired convergence towards zero within the prescribed time, fulfilling its expected role.
\end{itemize}
\begin{figure}[htbp]    
		\centering            
		\subfloat[MS-PTZGS]
		{
			\label{fig_xi_M}\includegraphics[width=0.24\textwidth]{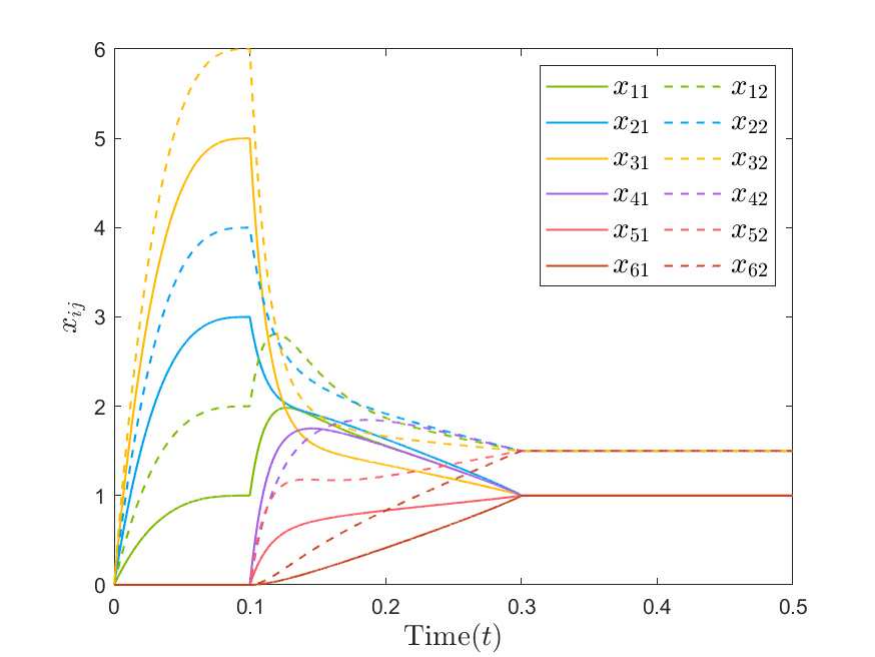}
		}
		\subfloat[SS-PTZGS]   
		{
			\label{fig_xi_S}\includegraphics[width=0.24\textwidth]{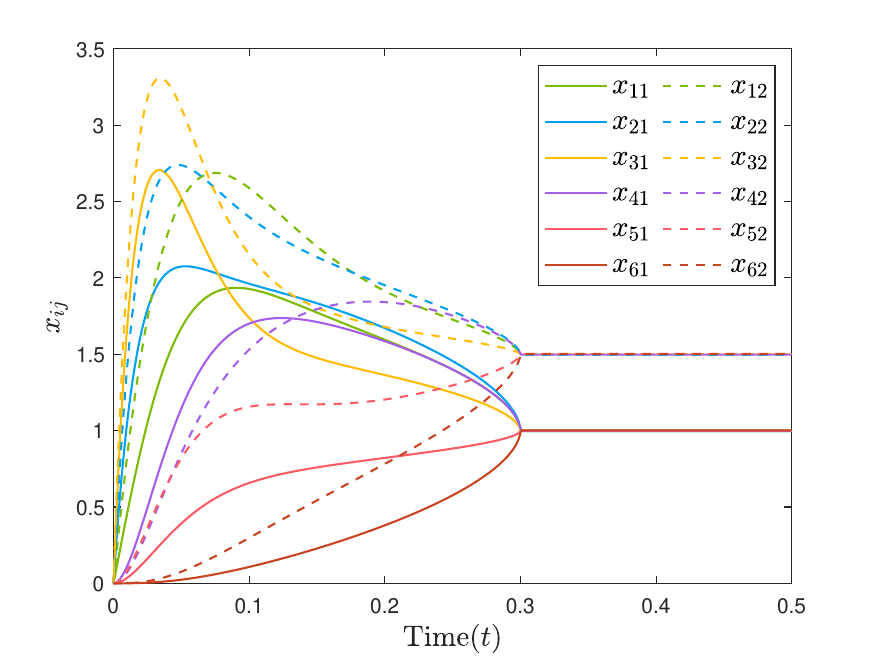}
		}
		\caption{Motion of each agent's state $x_{i}(t)$.}    
		\label{fig_xi}            
\end{figure}

\begin{figure}[htbp]    
		\centering            
		\subfloat[MS-PTZGS]
		{
			\label{fig_Re_M}\includegraphics[width=0.24\textwidth]{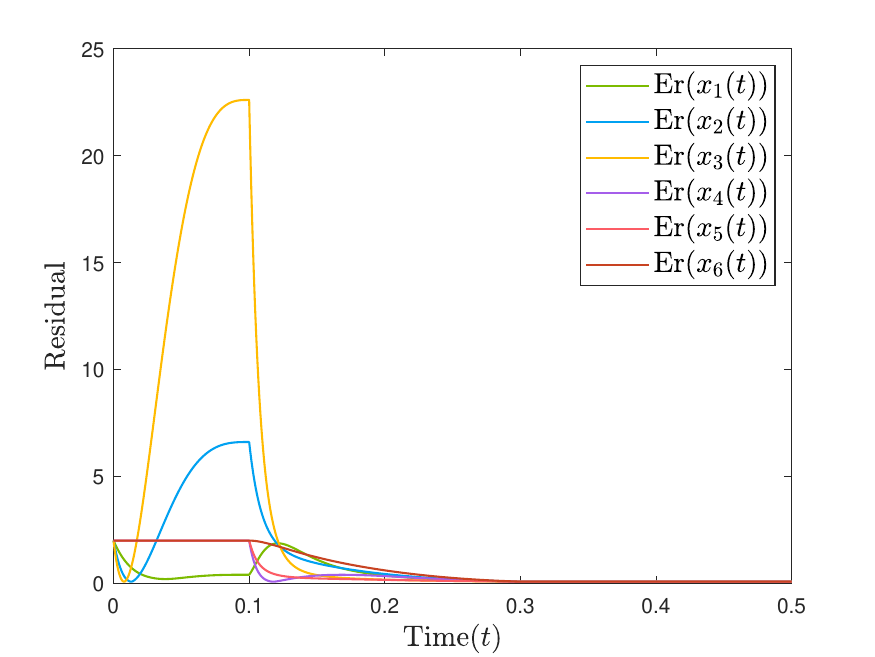}
		}
		\subfloat[SS-PTZGS]   
		{
			\label{fig_Re_S}\includegraphics[width=0.24\textwidth]{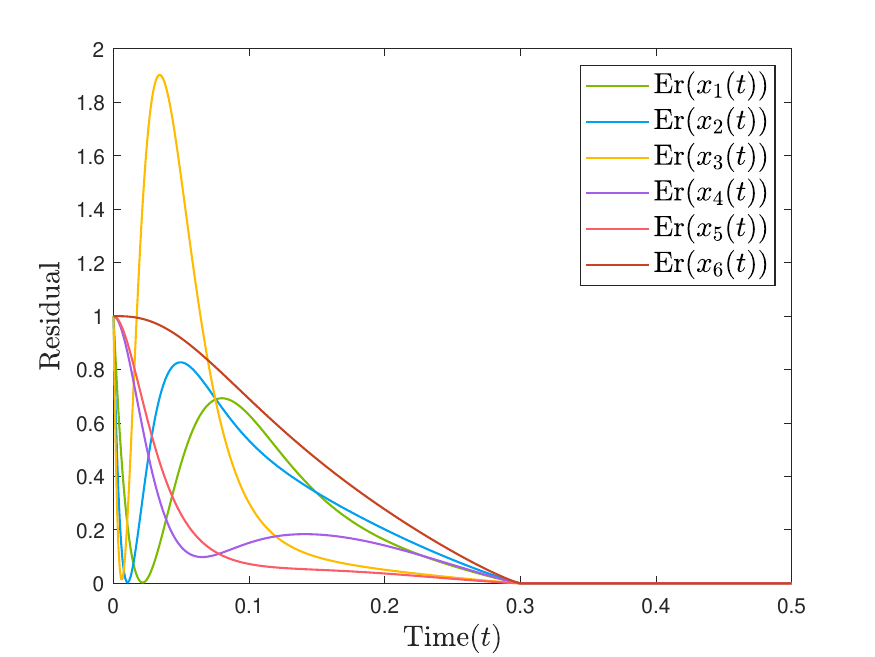}
		}
		\caption{Motion of each agent's normalized residual Er$(x_{i}(t))$.}    
		\label{fig_Re}            
\end{figure}

\begin{figure}[htbp]    
		\centering            
		\subfloat[MS-PTZGS]
		{
			\label{fig_si_M}\includegraphics[width=0.24\textwidth]{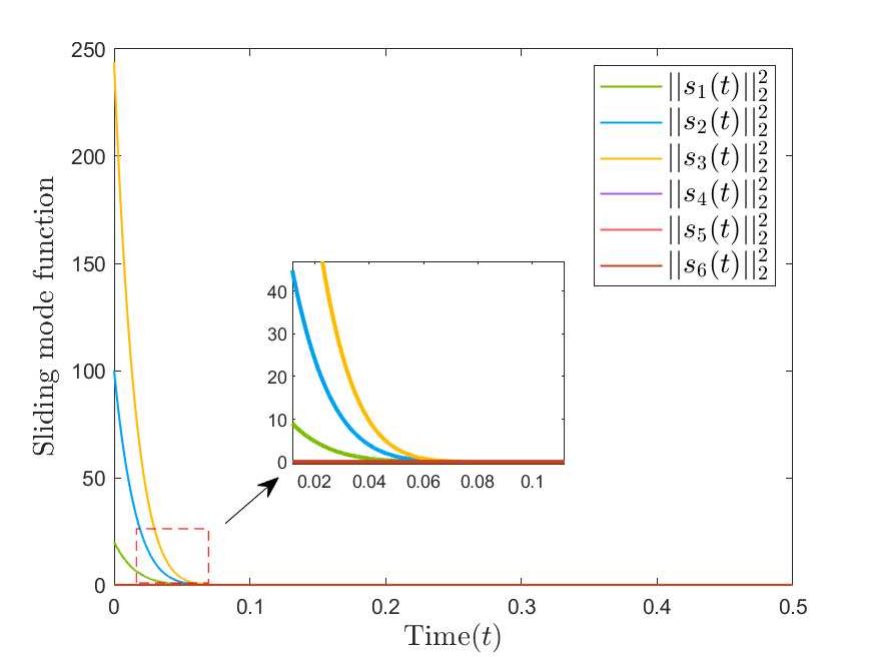}
		}
		\subfloat[SS-PTZGS]   
		{
			\label{fig_si_S}\includegraphics[width=0.24\textwidth]{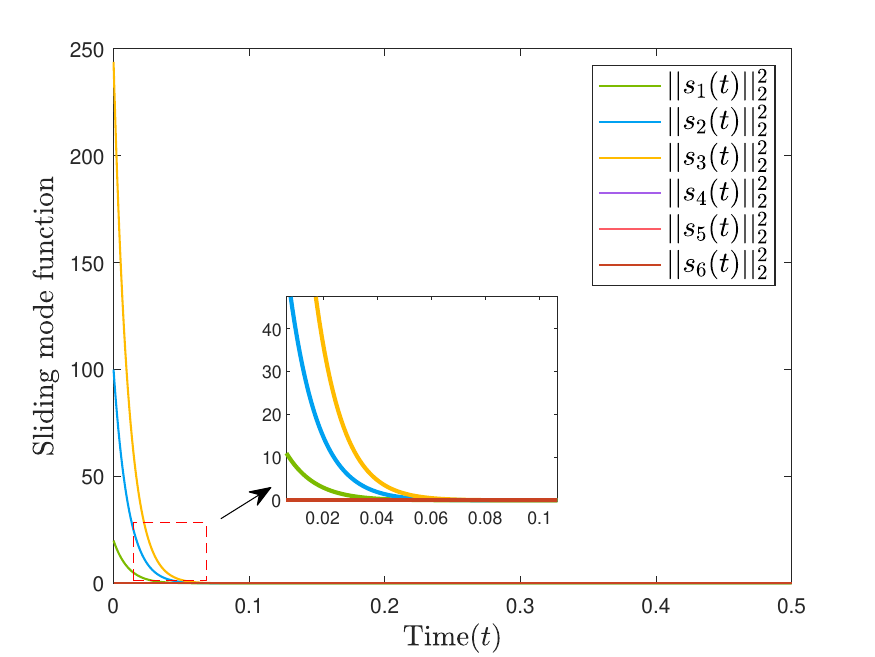}
		}
		\caption{Motion of each agent's sliding mode surface $s_{i}(t)$.}    
		\label{fig_si}            
\end{figure}

\begin{figure}[htbp]    
		\centering            
		\subfloat[MS-PTZGS]
		{
			\label{fig_y_M}\includegraphics[width=0.24\textwidth]{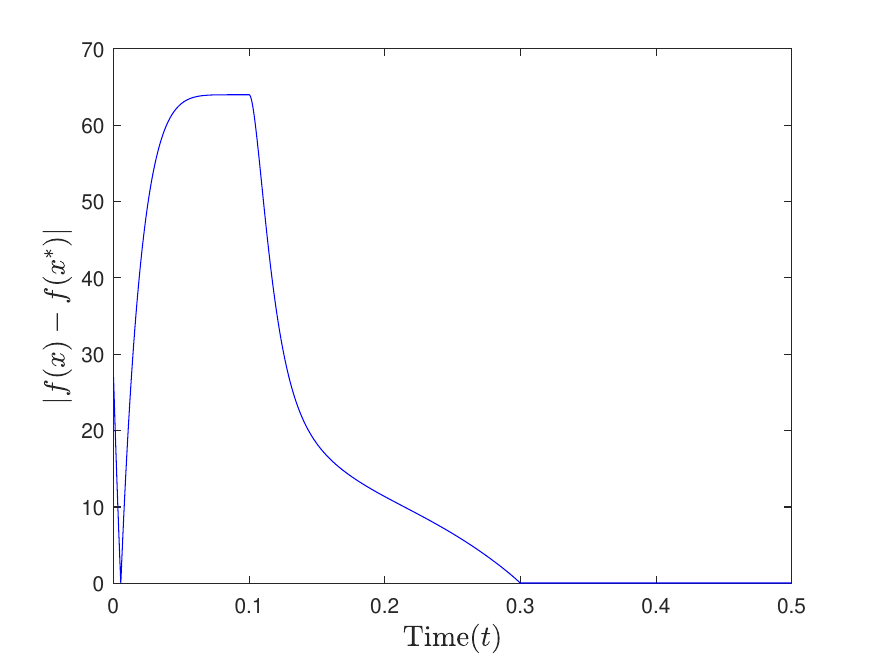}
		}
		\subfloat[SS-PTZGS]   
		{
			\label{fig_y_S}\includegraphics[width=0.24\textwidth]{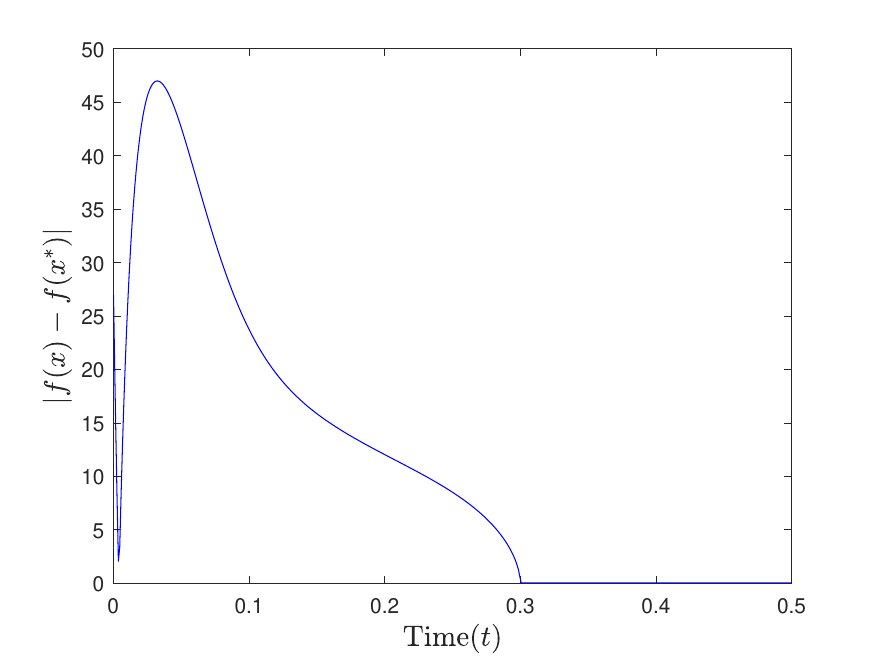}
		}
		\caption{Motion of the global function error $|f(\boldsymbol{x}(t))-f(x^{\ast})|$.}    
		\label{fig_y}            
	\end{figure}

\section{Conclusions}\label{Section 5}
Starting from a unified perspective of multi-stage and single-stage structures, this paper constructs two ZGS algorithms that guarantee prescribed-time convergence performance. The proposed MS-PTZGS avoids the requirement for the gradient sum to be zero and exhibits significant improvements in overall applicability and convergence performance compared to other ZGS algorithms. Furthermore, the introduced SS-PTZGS breaks through the constraints in ZGS algorithm design, addressing the issue that previous algorithms of this kind could not simultaneously achieve gradient sum reduction to zero and global convergence within a single time interval. Finally, the excellent performance of the proposed algorithms is validated through simulation.

\bibliographystyle{IEEEtran}
\bibliography{database}

\end{document}